\documentclass[psamsfonts]{amsart}
\usepackage{amssymb,amsfonts}
\usepackage[all,arc]{xy}
\usepackage{enumerate}
\usepackage{mathrsfs}
\usepackage{ stmaryrd }

\newcounter{qcounter}

\newenvironment{abcs}{
  \begin{list}
  {(\arabic{qcounter}) ~}
  {
    \usecounter{qcounter}
    \setlength{\itemsep}{0in}         
    \setlength{\topsep}{0in}
    \setlength{\partopsep}{0in}
    \setlength{\parsep}{0.05in}
    \setlength{\leftmargin}{0.285in}   
    \setlength{\rightmargin}{0in}
    \setlength{\itemindent}{0in}
    \setlength{\labelsep}{0in}
    \setlength{\labelwidth}{\leftmargin}
    \advance \labelwidth by-\labelsep
  }
}{
  \end{list}
}

\define\isoto{\xrightarrow{\sim}}
\define\onto{\twoheadrightarrow}
\define\coker{\mathrm{coker}}
\define\Spec{\mathrm{Spec}}
\define\Ch{\mathrm{Char}_\Lambda}
\define\Ft{\mathrm{Fitt}_\Lambda}
\define\Supp{\mathrm{Supp}}
\define\E{\mathrm{E}}

\define\otLam{\otimes_{\Lambda}}

\newcommand{\dia}[1]{\langle #1 \rangle}
\newcommand{\st}[1]{{\rm\bf #1}}

\newcommand{\cyc}{\mathbf{1} \boldsymbol{-} \boldsymbol{\zeta}}

\newcommand{\Z}{\mathbb{Z}}
\newcommand{\Q}{\mathbb{Q}}

\newcommand{\Zp}{\Z/{p\Z}}

\newcommand{\fH}{\mathfrak{H}}
\newcommand{\h}{\mathfrak{h}}

\newcommand{\fP}{\mathcal{P}}

\newcommand{\sO}{\mathcal{O}}

\newcommand{\tH}{\tilde{H}}
\newcommand{\Lam}{\Lambda}
\newcommand{\I}{\mathcal{I}}
\newcommand{\p}{\mathfrak{p}}
\newcommand{\m}{\mathfrak{m}}

\newcommand{\U}{\Upsilon}

\newcommand{\bnu}{\overline{\nu}}
\newcommand{\chii}{{\chi^{-1}}}

\newcommand{\LpN}{\Z_p[[\Z_{p,N}^\times]]}

\newcommand{\Hom}{\text{Hom}}
\newcommand{\Gal}{\text{Gal}}

\newcommand{\Ext}{\text{Ext}}

\newcommand{\HDM}{\tH_{\mathrm{DM}}}

\newcommand{\X}{\mathfrak{X}}

\newcommand{\zinf}{\{0,\infty\}}

\newtheorem{thm}{Theorem}[section]
\newtheorem{cor}[thm]{Corollary}
\newtheorem{prop}[thm]{Proposition}
\newtheorem{lem}[thm]{Lemma}
\newtheorem{conj}[thm]{Conjecture}

\newtheorem*{app}{Theorem \ref{application cor}}
\newtheorem*{main}{Theorem \ref{main}}

\theoremstyle{definition}
\newtheorem{defn}[thm]{Definition}

\theoremstyle{remark}
\newtheorem{rem}[thm]{Remark}

\makeatletter
\let\c@equation\c@thm
\makeatother
\numberwithin{equation}{section}

\title{Eisenstein Hecke algebras and conjectures in Iwasawa theory }

\author{Preston Wake}

\date{}

\begin{document}

\maketitle

\begin{abstract}
We formulate a weak Gorenstein property for the Eisenstein component of the $p$-adic Hecke algebra associated to modular forms. We show that this weak Gorenstein property holds if and only if a weak form of Sharifi's conjecture and a weak form of Greenberg's conjecture hold.
\end{abstract}

\section{Introduction}

In this paper, we study the relationship between the Iwasawa theory of cyclotomic fields and certain ring-theoretic properties of the Hecke algebra acting on modular forms. This continues work started in our previous paper \cite{wake}.

The philosophy of our work is that simplicity of the Iwasawa theory should correspond to simplicity of Hecke algebras. This philosophy comes from remarkable conjectures formulated by Sharifi \cite{sharifi}.

In \cite{wake}, we showed, under some assumptions, that if the Hecke algebra for modular forms is Gorenstein, then the plus part of the corresponding ideal class group is zero. In particular, we gave an example to show that this Hecke algebra is not always Gorenstein. 

Since the Hecke algebra is not always Gorenstein, it is natural to ask if there is a weaker ring-theoretic property that we can expect the Hecke algebra to have. In the present work, we formulate such a weaker property based on whether certain localizations of the Hecke algebra are Gorenstein. In a vague sense, we think of this condition as something like ``the obstructions to Gorenstein-ness are finite''.

We show that this weak Gorenstein property holds if and only if a weak form of Sharifi's conjecture and a  weak form of Greenberg's conjecture both hold. In particular, the weak Gorenstein property holds in every known example.

We make a few remarks before stating our results more precisely.

\subsubsection{Notation} In order to state our results more precisely, we introduce some notation. The notation coincides with that of \cite{wake}. 

 Let $p \ge 5$ be a prime and $N$ an integer such that $p \nmid \varphi(N)$ and $p \nmid N$. Let $\theta: (\Z/{Np \Z})^\times \to \overline{\Q}_p^\times$ be an even character and let $\chi=\omega^{-1} \theta$, where $\omega : (\Z/{Np \Z})^\times \to (\Z/p\Z)^\times \to \Z_p^\times$ denotes the Teichm\"{u}ller character. We assume that $\theta$ satisfies the same conditions as in \cite{kato-fukaya} -- namely that 1) $\theta$ is primitive, 2) if $\chi |_{(\Zp)^\times}=1$, then $\chi |_{(\Z/{N\Z})^\times}(p) \ne 1$,  3) if $N=1$, then $\theta \ne \omega^2$. 
 
 A subscript $\theta$ or $\chi$ will denote the eigenspace for that character for the  $(\Z/{Np \Z})^\times$-action (see Section \ref{notation}).
 
Let  $\Lam=\Z_p[[\Z_{p,N}^\times]]_\theta$ be the Iwasawa algebra, where $\Z_{p,N}^\times = \Z_p^\times \times (\Z/N\Z)^\times$. Let $\m_\Lambda$ be the maximal ideal of $\Lambda$.
 
Let $\fH$ (resp. $\h$) be the $\theta$-Eisenstein component of the Hecke algebra for $\Lambda$-adic modular forms (resp. cusp forms). Let $\I$ (resp. $I$) be the Eisenstein ideal of $\fH$ (resp. $\h$), and let $I_\fH \supset \I$ be the preimage of $I$ in $\fH$. Let $H$ be the cohomology group on which $\h$ acts (see Section  \ref{hecke def}).

Let $\Q_\infty=\Q(\zeta_{Np^\infty})$; let $M$ be the maximal abelian $p$-extension of $\Q_\infty$ unramified outside $Np$, and let $L$ be  the maximal abelian $p$-extension of $\Q_\infty$ unramified everywhere. Let $\X=\Gal(M/\Q_\infty)$ and $X=\Gal(L/\Q_\infty)$.

\subsection{Statement of Results} 

\subsubsection{Weakly Gorenstein Hecke algebras} We define what it means for the Hecke algebras $\h$ and $\fH$ to be weakly Gorenstein. In the case of $\h$, the definition comes from a condition that appears in work of Fukaya-Kato on Sharifi's conjecture \cite[Section 7.2.10]{kato-fukaya}, and is related to a condition that appears in work of Sharifi \cite{sharifi eisen}.

\begin{defn}
We say that $\h$ is {\em weakly Gorenstein} if $\h_\p$ is Gorenstein for every prime ideal $\p \in \Spec(\h)$ of height $1$ such that $I \subset \p$.

We say that $\fH$ is {\em weakly Gorenstein} if $\fH_\p$ is Gorenstein for every prime ideal $\p \in \Spec(\fH)$ of height $1$ such that $I_\fH \subset \p$.
\end{defn}
In general, neither the algebra $\h$ nor the algebra $\fH$ is Gorenstein. However, we conjecture that they are both weakly Gorenstein.
\begin{conj}
The Hecke algebras $\h$ and $\fH$ are weakly Gorenstein.
\end{conj}

\subsubsection{Relation to ideal class groups} These ring-theoretic properties of Hecke algebras are related to ideal class groups via Sharifi's conjecture \cite{sharifi}. Sharifi has constructed a map $$\U: X_\chi(1) \to H^-/IH^-$$ which he conjectures to be an isomorphism. 

A weaker conjecture is that $\U$ is a pseudo-isomorphism -- recall that a morphism of $\Lambda$-modules is called a  pseudo-isomorphism if its kernel and cokernel are both finite. If $\U$ is a pseudo-isomorphism, then $\h$ is weakly Gorenstein if and only $X_\chi$ is pseudo-cyclic (cf. Section \ref{h section} below). We have the following analogous result for $\fH$. In the statement of the theorem, $\xi_\chi$ is a characteristic power series for $X_\chi(1)$ as a $\Lambda$-module.

\begin{thm} 
\label{main}
Consider the following conditions.
\abcs
\item  $\fH$ is weakly Gorenstein
\item $\coker(\U)$ is finite
\item $X_\theta/ \xi_\chi X_\theta$ is finite. 
\endabcs
Condition (1) holds if and only if conditions (2) and (3) both hold.
\end{thm}
\begin{rem}\label{remark}
Note that if $X_\chi = 0$, then all three conditions hold trivially. Indeed, if $X_\chi = 0$, then $\fH=\Lambda$, the domain and codomain of $\Upsilon$ are $0$, and $ \xi_\chi$ is a unit (cf. \cite[Remark 1.3]{wake}).
\end{rem}
\begin{rem}
The conditions (2) and (3) are conjectured to hold in general (see Section \ref{relation}). In particular, they hold in all known examples. 
\end{rem}
\begin{rem}
Condition (2) is equivalent to the condition that $\U$ is an injective pseudo-isomorphism (see Proposition \ref{sha prop}).
\end{rem}
\begin{rem}
Condition (3) is strange: $\xi_\chi$ is the opposite of the usual $p$-adic zeta function that is related to $X_\theta$. That is, $X_\theta$ is annihilated by $\xi_\chii$, and not (at least not for any obvious reason) by $\xi_\chi$ .
\end{rem}

The proof of Theorem \ref{main} will be given in Section \ref{proof section}.

\subsubsection{Strong and weak versions of Sharifi's conjecture} One consequence of Sharifi's conjecture is that $X_\chi(1) \cong H^-/IH^-$ as $\Lambda$-modules. Since $X_\chi$ has no $p$-torsion, this would imply that $H^-/IH^-$ has no $p$-torsion, which Sharifi explicitly conjectures in \cite[Remark, pg. 51]{sharifi}. 

A theorem of Ohta implies that if $\fH$ is Gorenstein, then $X_\chi(1) \cong H^-/IH^-$ (cf. Theorem \ref{H gor implies U iso thm} below). Moreover, Ohta also also proves that $\fH$ is Gorenstein under a certain hypothesis (\cite[Theorem I]{comp2}, for example). Sharifi used this as evidence for his conjecture (\cite[Proposition 4.10]{sharifi}). 

Since it is now known that $\fH$ is not always Gorenstein (\cite[Corollary 1.4]{wake}), one may wonder if Sharifi's conjecture should be weakened to the statement ``$\U$ is pseudo-isomorphism'' (cf. Conjecture \ref{weak sharifi conj} below). Fukaya and Kato  (\cite{kato-fukaya}) have partial results on this version of the conjecture. When neither $\h$ nor $\fH$ is Gorenstein, we know of no evidence for Sharifi's conjecture that $\U$ is an isomorphism (and not just a pseudo-isomorphism); we hope that our next result can be used to provide evidence. This result concerns a module $H^-/I\HDM^-$. As explained in Section \ref{hecke section} below, $H^-/I\HDM^-$ measures how much the ring $\fH$ is ``not Gorenstein''. 

For a finitely generated $\Lambda$-module $M$, let $$d_{\m_\Lambda}(M)=\dim_{\Lambda/\m_\Lambda}(M/\m_\Lambda M).$$
Note that $d_{\m_\Lambda}(M)$ is the minimal number of generators of $M$ as a $\Lambda$-module.

\begin{thm}\label{application cor}
Assume that $X_\theta \ne 0$ and that $\h$ is weakly Gorenstein. Then we have 
$$d_{\m_\Lambda}(H^-/I\HDM^- ) \ge d_{\m_\Lambda}(X_\chi)$$
with equality if and only if $\U$ is an isomorphism.

If, in addition, $\# (X_\theta)= \# (\Lambda/\m_\Lambda)$, then $\U$ is an isomorphism if and only if 
$$
\#(H^-/I\HDM^-)=\# (\Lambda/\m_\Lambda)^{d_{\m_\Lambda}(X_\chi)}.
$$
\end{thm}

This theorem may be used to provide evidence for Sharifi's conjecture that $\U$ is an isomorphism in two ways. The first way is philosophical: although the ring $\fH$ is not always Gorenstein, we may like to believe that $\fH$ is ``as close to being Gorenstein as possible''. This translates to the belief that $H^-/I\HDM^-$ is as small as possible; the theorem says that $H^-/I\HDM^-$ is smallest when $\U$ is an isomorphism.

The second way is a method for providing computational evidence: Theorem \ref{application cor} may  allow one to compute examples where $\U$ is an isomorphism but where $\fH$ is not Gorenstein. We now outline a scheme for doing this. First, one finds an imaginary quadratic field with non-cyclic $p$-class group; this provides a character $\chi$ of order $2$ such that $X_\theta \ne 0$ and such that $\fH$ is not Gorenstein (cf. \cite[Corollary 1.4]{wake}). However, $\h$ will be weakly Gorenstein by Lemma \ref{h weakly gor lemma} below (or else we have found a counterexample to a famous conjecture!). The assumptions for Theorem \ref{application cor} are then satisfied. Then, if one can compute $H^-/I\HDM^-$ and $X_\chi$ sufficiently well, one can verify that $d_{\m_\Lambda}(H^-/I\HDM^- ) = d_{\m_\Lambda}(X_\chi).$

The proof of Theorem \ref{application cor} will be given in Section \ref{application}.

\subsection{Relation to known results and conjectures}\label{relation} Our results are related to previous results and conjectures of various authors, including Fukaya-Kato, Greenberg, Ohta, Sharifi, Skinner-Wiles and the present author. In the main text, we try to survey these results and conjectures. However, since this is an area with many conjectures, and many of the results are about the interrelation of the conjectures or proofs of special cases of the conjectures, the reader may find it difficult to see what is known, and what is unknown, and what exactly is conjectured.

In this section, we try to write down the conjectures and results in a compact but clear fashion. This involves creating an unorthodox naming convention, which we hope will aid in understanding the connections between the statements. The reader may wish to skip this section, and use it as a reference when reading the main text.

\subsubsection{Naming convention} We use \st{C(Y)} to denote a conjecture about Y, \st{Q(Y)} to denote a question about about Y (a statement that is not conjectured to be true or false), and \st{A(Y)} to denote an assumption about Y (a statement that is {\bf known} to be false in general).

Numbered statement are listed in increasing order of logical strength. For example, \st{Q(Y II)} is a questionable statement about Y that implies \st{C(Y I)}, a conjectural statement about Y.

\subsubsection{Finiteness conditions} We consider the following statements about finiteness and cyclicity of class groups.
\begin{description}
\item[C(Fin I)]  $X_\theta/ \xi_\chi X_\theta$ is finite.
\item[C(Fin II)] $X_\theta$ is finite. 
\item[A(Fin III)] $X_\theta=0$.
\item[A(Fin IV)] $\X_\theta=0$.
\item[C(Cyc I)] $\X_\theta \otimes_{\Z_p} \Q_p$ is cyclic as a $\Lambda \otimes_{\Z_p} \Q_p$-module.
\item[A(Cyc II)] $\X_\theta$ is cyclic as a $\Lambda$-module.
\item[C(Cyc' I)] $X_\chi \otimes_{\Z_p} \Q_p$ is cyclic as a $\LpN_\chi \otimes_{\Z_p} \Q_p$-module.
\item[A(Cyc' II)] $X_\chi$ is cyclic as a $\LpN_\chi$-module.
\end{description}
There are implications \st{A(Fin III)} $\implies$ \st{A(Cyc II)} and \st{C(Fin II)} $\implies$ \st{C(Cyc I)}. In the case $N=1$, \st{A(Fin III)} is actually a conjecture, known as the Kummer-Vandiver conjecture. The conjectures \st{C(Fin II)}, \st{C(Cyc I)}  and \st{C(Cyc' I)} are due to Greenberg \cite[Conjecture 3.5]{greenberg}. Note that there is no relation between the conjectures \st{C(Cyc' I)} and \st{C(Cyc I)} for our fixed choices of $\chi$ and $\theta$ (the modules are {\bf not} adjoint -- see Proposition \ref{X and X prop}).

As far as we know, the Conjecture \st{C(Fin I)} has never been considered before.

\subsubsection{Gorenstein conditions} We consider the following statements about Hecke algebras. 
\begin{description}
\item[C($\h$ I)] $\h$ is weakly Gorenstein.
\item[A($\h$ II)] $\h$ is Gorenstein.
\item[C($\fH$ I)] $\fH$ is weakly Gorenstein.
\item[A($\fH$ II)] $\fH$ is Gorenstein.
\end{description}
The fact that $\h$ is not always Gorenstein is can be deduced from results of Ohta (following ideas of Kurihara \cite{kurihara} and Harder-Pink \cite{harder-pink}, who consider the case $N=1$). Ohta proves the implication \st{A($\h$ II)} $\implies$ \st{A(Cyc' II)} (\cite[Corollary 4.2.13]{comp2}, for example).

The fact that $\fH$ is not always Gorenstein is \cite[Corollary 1.4]{wake}. The weakly Gorenstein conjectures are ours (although this paper shows that they are the consequence of conjectures by other authors).

\subsubsection{Conjectures of Sharifi type} We consider the following versions of Sharifi's conjecture. They concern maps $\varpi$ and $\U$ that were defined by Sharifi.
\begin{description}
\item[C($\U$ I)] $\coker(\U)$ is finite.
\item[C($\U$ II)] $\U$ is an isomorphism.
\item[C(S. I)] The maps $\U$ and $\varpi$ are pseudo-isomorphisms.
\item[C(S. II)] The maps $\U$ and $\varpi$ are inverse isomorphisms modulo torsion.
\item[C(S. III)] The maps $\U$ and $\varpi$ are inverse isomorphisms.
\end{description}
Note that \st{C(S. I)} implies \st{C($\U$ I)}, and  that \st{C(S. III)} is equivalent to \st{C(S. II)} + \st{C($\U$ II)}.

See \cite[Conjectures 4.12, 5.2 and 5.4]{sharifi} for the original statements of the conjecture and \cite[Section 7.1]{kato-fukaya} for some modified statements.
 
\subsubsection{A question about zeta functions} We consider the following statement about $p$-adic zeta functions that appears in \cite{kato-fukaya}.

\begin{description}
\item[Q($\xi$)] The factorization of $\xi_\chi$ in $\Lambda$ has no prime element occurring with multiplicity $>1$.
\end{description}
This statement holds in any known example (see \cite[pg. 12]{greenberg}). It is the author's impression that this statement is believed to hold in general, but that there is not enough evidence to call it a conjecture.

\subsubsection{Relations between the conditions}

Fukaya and Kato have made recent progress towards Sharifi's conjecture. They show the implications \st{C($\h$ I)} $\implies$ \st{C(S. I)} and \st{Q($\xi$)} $\implies$ \st{C(S. II)} \cite[Theorem 7.2.6]{kato-fukaya}. They also show that if \st{Q($\xi$)} and at least one of \st{A($\h$ II)} or \st{A($\fH$ II)} hold, then \st{C(S. III)} holds \cite[Corollary 7.2.7]{kato-fukaya}. Moreover, it can be shown that, if \st{C($\U$ I)}, then \st{C(Cyc' I)} is equivalent to \st{C($\h$ I)} (cf. Section 5.1 below). Therefore, their results imply that \st{C($\h$ I)} is equivalent to \st{C(S. I)} +  \st{C(Cyc' I)}.

Sharifi, using results of Ohta \cite{cong}, has shown that \st{A($\fH$ II)} $\implies$ \st{C($\U$ II)} \cite[Proposition 4.10]{sharifi}. As far as we know, there are no results on \st{C(S. III)} when neither \st{A($\h$ II)} nor \st{A($\fH$ II)} hold.

Ohta has also shown that \st{A(Fin IV)} $\implies$ \st{A($\fH$ II)} \cite{cong}. Similar results were obtained by Skinner-Wiles by a different method \cite{skinner-wiles}.

In our previous work we showed that \st{C($\U$ II)} and \st{A(Fin III)} together imply  \st{A($\fH$ II)}, and moreover that, if $X_\chi \neq 0$, then  \st{A($\fH$ II)} implies \st{A(Fin III)} \cite{wake}. 

The main result of this paper is that \st{C($\fH$ I)} is equivalent to \st{C($\U$ I)}+\st{C(Fin I)}.

\subsection{Acknowledgments} We owe a great debt to Romyar Sharifi for his interest in, and careful reading of, our previous paper \cite{wake}. It was he who asked what happens when $X_\theta$ is finite, and suggested that the answer could be found by the methods of \cite{wake}. We thank him for his interest and insight.

We also thank Takako Fukaya for answering our questions, Ralph Greenberg for his interest in this problem, and Chris Skalit for helpful discussions about commutative algebra. 

This work was done while the author received partial support from NSERC of Canada. We thank them for the support.

Some of the ideas for this work were developed while the author enjoyed a short stay at the Universit\'e Paris-Sud. We thank the university for their hospitality, and Olivier Fouquet for his invitation and for many interesting discussions during the stay.

We thank the anonymous referee for a detailed and encouraging report that contained many helpful comments and corrections. We especially appreciate a suggestion about the notation, which has made the paper much more readable.

Finally, we thank Professor Kazuya Kato. It was he who introduced us to this interesting area and who encouraged us to write up these results. We thank him for his questions, comments and corrections, great and small, and for his patient advice and wisdom.

\subsection{Conventions}\label{notation} If $\phi: G \to \overline{\Q}_p^\times$ is a character of a group $G$, we let $\Z_p[\phi]$ denote the $\Z_p$-algebra generated by the values of $\phi$, on which $G$ acts through $\phi$. If $M$ is a $\Z_p[G]$-module, denote by $M_\phi$ the $\phi$-eigenspace:
$$
M_\phi = M \otimes_{\Z_p[G]} \Z_p[\phi].
$$

For a field $K$, let  $G_K = \Gal(\overline{K}/K)$ be the absolute Galois group. For a $G_\Q$-module $M$, let $M^+$ and $M^-$ denote the eigenspaces of complex conjugation. 

We fix a system of primitive $Np^r$-th roots of unity $(\zeta_{Np^r})$ with the property that $\zeta_{Np^{r+1}}^p=\zeta_{Np^r}$.

\section{Conjectures in Iwasawa theory}

\subsection{Iwasawa theory of cyclotomic fields}\label{cyc field section} We review some important results from the classical Iwasawa theory of cyclotomic fields. Nice references for this material include \cite{greenberg}, \cite{greither}, and \cite{washington}.

\subsubsection{Class groups and Galois groups}   The main object of study is the inverse limit of the $p$-power torsion part of the ideal class group $\mathrm{Cl}(\Q(\zeta_{Np^r}))$. By class field theory, there is an isomorphism 
$$
X \cong \varprojlim \mathrm{Cl}(\Q(\zeta_{Np^r}))\{p\},
$$
where, as in the introduction, $X=\Gal(L/\Q_\infty)$ where $L$ is the maximal abelian pro-$p$-extension of $\Q_\infty$ unramified everywhere, and where $(-)\{p\}$ denotes the $p$-Sylow subgroup.

A closely related object is $\X=\Gal(M/\Q_\infty)$, where $M$ is the maximal abelian pro-$p$-extension of $\Q_\infty$ unramified outside $Np$. We will explain the relation between $X$ and $\X$ below. 

\subsubsection{Iwasawa algebra} The natural action of $\Gal(\Q(\zeta_{Np^r})/\Q)$ on  $\mathrm{Cl}(\Q(\zeta_{Np^r}))\{p\}$ makes $X$ a module over the group ring $\varprojlim \Z_p[\Gal(\Q(\zeta_{Np^r})/\Q)]$. 

We fix a choice of isomorphism $\Gal(\Q(\zeta_{Np^r})/\Q) \cong (\Z/Np^r\Z)^\times$, and this induces an isomorphism $\varprojlim \Z_p[\Gal(\Q(\zeta_{Np^r})/\Q)] \cong \Z_p[[\Z_{p,N}^\times]]$, where recall that we define $\Z_{p,N}^\times = \Z_p^\times \times (\Z/N\Z)^\times$. Note that the  surjection $\Z_{p,N}^\times \to (\Z/Np\Z)^\times$ splits canonically. We use this to identify $\Z_{p,N}^\times$ with  $\Gamma \times (\Z/Np\Z)^\times$, where $\Gamma$ is the torsion-free part of $\Z_{p,N}^\times$ (note that $\Gamma \cong \Z_p$).

The ring $\Z_p[[\Z_{p,N}^\times]]$ is, in general, a product of rings. To simplify things, we consider only a particular eigenspace for the action of the torsion subgroup $(\Z/Np\Z)^\times$ of $\Z_{p,N}^\times$. We define $\Lambda = \Z_p[[\Z_{p,N}^\times]]_{\theta}$. There are isomorphisms $\Lam \cong \sO[[\Gamma]] \cong \sO[[T]]$, where $\sO$ is the $\Z_p$-algebra generated by the values of $\theta$. Note that $\sO$ is the valuation ring of a finite extension of $\Q_p$, and so $\Lambda$ is a noetherian regular local ring of dimension 2 with finite residue field.

\subsubsection{The operators $\tau$ and $\iota$} We introduce some operations $\iota$ and $\tau$ on the rings $\Z_p[[\Z_{p,N}^\times]]$ and $\Lambda$, and related functors $M \mapsto M^\#$ and $M \mapsto M(r)$. This is a technical part, and the reader may wish to ignore any instance of these on a first reading.

Let $\iota \colon \Z_p[[\Z_{p,N}^\times]] \to \Z_p[[\Z_{p,N}^\times]]$ by the involution given by $c \mapsto c^{-1}$ on $\Z_{p,N}^\times$. Let $\tau \colon \Z_p[[\Z_{p,N}^\times]] \to \Z_p[[\Z_{p,N}^\times]]$ be the morphism induced by $[ c] \mapsto \bar{c}[ c ]$ for $c \in \Z_{p,N}^\times$, where $[c] \in \Z_p[[\Z_{p,N}^\times]]$ is the group element and where $\bar{c} \in \Z_p^\times$ is the projection of $c$.

Note that $\iota$ and $\tau$ do not commute, but $\iota \tau = \tau^{-1} \iota$. In particular, $\tau^r \iota$ is an involution for any $r \in \Z$.

For a $\Z_p[[\Z_{p,N}^\times]]$-module $M$, we let $M^\#$ (resp. $M(r)$) be the same abelian group with $\Z_p[[\Z_{p,N}^\times]]$-action changed by $\iota$ (resp. $\tau^r$). Note that the functors $M \mapsto M^\#$ and $M \mapsto M(r)$ are exact.

\subsubsection{$p$-adic zeta functions and characteristic ideals} 
We define $\xi_\chii, \xi_\chi \in \Lambda$ to be generators of the principal ideals $\Ch(X_\chii^\#(1))$ and $\Ch(X_\chi(1))$ respectively (see the Appendix for a review of characteristic ideals). 

The Iwasawa Main Conjecture (now a theorem of Mazur and Wiles \cite{mazur-wiles}) states that (a certain choice of) $\xi_\chii$ and $\xi_\chi$ can be constructed by $p$-adically interpolating values of Dirichlet $L$-functions. 

\begin{rem}\label{defn of xi rem} In our previous work \cite{wake}, we viewed $X_\chi$ and $X_\chii$ as $\Lambda$-modules via the isomorphisms
$$
\tau: \Z_p[[\Z^\times_{p,N}]]_{\chi} \isoto \Lambda ; \ \iota\tau: \Z_p[[\Z^\times_{p,N}]]_{\chii} \isoto \Lambda.
$$
We learned from the referee that this is an unusual choice of notation, and so we have adopted the above convention, which we learned is more standard. 

The element $\xi$ of \cite{wake} is the $\xi_\chi$ of this paper. However, the element denoted $\xi_\chii$ in \cite{wake} would be denoted $\iota \tau \xi_\chii$ in this paper. We hope this doesn't cause confusion.
\end{rem}

\subsubsection{Adjoints}\label{adjoints section1} For a finitely generated $\Lambda$-module $M$, let $\E^i(M) = \Ext^i_{\Lambda}(M,\Lambda)$. These are called the {\em (generalized) Iwasawa adjoints} of $M$.

This theory is important to us because of the following fact, which is well-known to experts.

\begin{prop}\label{X and X prop}
The $\LpN$-modules $X_\chii$ and $\X_\theta$ are both torsion and have no non-zero finite submodule, and we have 
$$
\X_\theta \cong \E^1(X_\chii(-1)).
$$
In particular, we have $\Ch(\X_\theta)=(\xi_\chii)$ as ideals in $\Lambda$.
\end{prop}
\begin{proof}
The first sentence is explained in \cite[Corollary 4.4]{wake}. The second sentence follows from the fact that for any finitely generated, torsion ${\Lambda}$-module $M$, there is a pseudo-isomorphism $E^1(M) \to M^\#$ (\cite[Proposition 5.5.13, pg. 319]{N-S-W}). Since $\Ch(-)$ is a pseudo-isomorphism invariant, we have 
$$
\Ch(\E^1(X_\chii(-1))) = \Ch((X_\chii(-1))^\#) = \Ch(X_\chii^\#(1)) = (\xi_\chii),
$$
and so the second sentence follows from the first.
\end{proof}

\subsubsection{Exact sequence} There is a exact sequence
\begin{equation}\label{euxxc2}
\Lambda/\xi_\chii \to \X_\theta \to X_\theta \to 0.
\end{equation}
coming from class field theory and Coleman power series (cf. eg. \cite[Sections 3 and 5]{wake}).
From Proposition \ref{X and X prop} and the fact that $\Lambda/\xi_\chii$ has no finite submodule, we can see that $X_\theta$ is finite (resp. zero) if and only if the leftmost arrow in (\ref{euxxc2}) is injective (resp. an isomorphism).

\subsection{Finiteness and cyclicity of class groups} We discuss some statements of finiteness and cyclicity of ideal class groups.

\subsubsection{Kummer-Vandiver conjecture} We first consider the case $N=1$.

\begin{conj}[Kummer-Vandiver] \label{kummer-vandiver}
Assume $N=1$. Then $X^+=0$.
\end{conj}

\begin{lem}\label{van lem}
Assume $N=1$. If Conjecture \ref{kummer-vandiver} is true, then $\X_\theta$ and $X_\chii$ are cyclic.
\end{lem}
\begin{proof}
If $X_\theta=0$, we see from (\ref{euxxc2}) that $\X_\theta$ is cyclic. We wish to show that $X_\chii$ is cyclic. It is enough to show that $X_\chii(-1)$ is cyclic, and we claim that this follows from the fact that  $\X_\theta$ is cyclic by Proposition \ref{X and X prop} and standard arguments from \cite{N-S-W}. Indeed, we have isomorphisms
$$
X_\chii(-1) \cong \E^1(\E^1(X_\chii(-1))) \cong \E^1(\X_\theta)
$$
coming from \cite[Proposition 5.5.8 (iv), pg. 316]{N-S-W} and Proposition \ref{X and X prop}, respectively. So it is enough to show that $\E^1(\X_\theta)$ is cyclic whenever $\X_\theta$ is. But this is clear from \cite[Proposition 5.5.3 (iv), pg. 313]{N-S-W}, which says that the projective dimension of $\X_\theta$ is $1$; if $\X_\theta$ is generated by one element, then there is exactly one relation, and the dual of the resulting presentation gives a cyclic presentation of $\E^1(\X_\theta)$.
\end{proof}

\subsubsection{Greenberg's conjecture}\label{greenberg section} For general $N>1$, there are examples where $X_\chi$ is not cyclic, and so $X^+$ is not always zero. However, it may still be true that $X^+$ is finite \cite[Conjecture 3.4]{greenberg}.

\begin{conj}[Greenberg]\label{green conj}
The module $X^+$ is finite.
\end{conj}

The following lemma may be proved in the same manner as Lemma \ref{van lem}.

\begin{lem}\label{green lem}
The following are equivalent:
\abcs
\item $X_\theta$ is finite.
\item The map $\Lambda/\xi_\chii \otimes_{\Z_p} \Q_p \to \X_\theta \otimes_{\Z_p} \Q_p$ induced by the map in (\ref{euxxc2}) is an isomorphism.
\endabcs
\end{lem}

\section{Hecke algebras and modular forms} 

In this section, we introduce Hecke algebras for modular forms to the story. 

\subsection{Hecke algebras and Eisenstein ideals} \label{hecke def} We introduce the objects from the theory of modular forms that we will need. See \cite{wake2} for a more detailed treatment of this theory.

\subsubsection{Modular curves and Hecke operators} 	Let $Y_1(Np^r) /\Q$ be the moduli space for pairs $(E,P)$, where $E/\Q$ is an elliptic curve and $P \in E$ is a point of order $Np^r$. Let $X_1(Np^r)/\Q$ be the compactification of $Y_1(Np^r)$ by adding cusps. 

There is an action of $(\Z/Np^r\Z)^\times$ on $Y_1(Np^r)$ where $a \in (\Z/Np^r\Z)^\times$ acts by $a(E,P)= (E,aP)$. This is called the action of diamond operators $\dia{a}$. There are also Hecke correspondences $T^*(n)$ on $Y_1(Np^r)$ and $X_1(Np^r)$. We consider these as endomorphisms of the cohomology.

We define the Hecke algebra of $Y_1(Np^r)$ to be the algebra generated by the $T^*(n)$ for all integers $n$ and the $\dia{a}$ for all $a \in (\Z/Np^r\Z)^\times$. We define the Eisenstein ideal to be the ideal of the Hecke algebra generated by $1-T^*(l)$ for all primes $l | Np$ and by $1-T^*(l)+l\dia{l}^{-1}$ for all primes $l \nmid Np$.
\subsubsection{Ordinary cohomology}
Let $$
H'=\varprojlim H^1(\overline{X}_1(Np^r),\Z_p)^{\mathrm{ord}}_{\theta}$$
and
$$
\tH'= \varprojlim H^1(\overline{Y}_1(Np^r),\Z_p)^{\mathrm{ord}}_{\theta},$$
where the superscript ``$\mathrm{ord}$'' denotes the ordinary part for the dual Hecke operator $T^*(p)$, and the subscript refers to the eigenspace for the diamond operators.

\subsubsection{Eisenstein parts} Let $\h'$ (resp. $\fH'$) be the algebra of dual Hecke operators acting on $H'$ (resp. $\tH'$). Let $I$ (resp. $\I$) be the Eisenstein ideal of $\h'$ (resp. $\fH'$).  Let $\fH$ denote the {\em Eisenstein component} $\fH = \fH'_\mathfrak{m}$, the localization at the unique  maximal ideal $\mathfrak{m}$ containing $\I$. We can define the Eisenstein component $\h$  of $\h'$ analogously. Let $\tH=\tH' \otimes_{\fH'} \fH$ and $H=H' \otimes_{\h'} \h$ be the Eisenstein components.

There is a natural surjection $\fH \onto \h$ by restriction. Let $I_\fH \subset \fH$ be the kernel of the composite map $\fH \onto \h \onto \h/I$. Note that $\I \subsetneq
I_\fH$.

\subsection{Properties of the Hecke modules} We first recall some properties of the Hecke modules $\tH$ and $H$ and Hecke algebras $\fH$ and $\h$. See \cite[Section 6]{kato-fukaya} for a simple and self-contained exposition of this. 

\subsubsection{Control theorem} There are natural maps $\Lambda \to \h$ and $\Lambda \to \fH$ given by diamond operators. It is a theorem of Hida that these maps are finite and flat. In particular, $\h$ and $\fH$ are noetherian local rings of dimension 2 with (the same) finite residue field. 

Let $\h^\vee$ (resp. $\fH^\vee$) denote the $\h$-module (resp. $\fH$-module) $\Hom_\Lambda(\h,\Lambda)$  (resp. $\Hom_\Lambda(\fH,\Lambda)$). We will call these the {\em dualizing modules} for the respective algebras.

\subsubsection{Eichler-Shimura isomorphisms} Ohta (\cite[Section 4.2]{comp2}, for example) has proven theorems on the Hecke module structure of $\tH$ and $H$. See \cite{wake2} for a different approach. The main result we need is the following, which appears in this form in \cite[Section 6.3]{kato-fukaya}.

\begin{thm}\label{e-s}
There are isomorphisms of Hecke-modules $H^+ \cong \h$, $H^- \cong \h^\vee$ and $\tH^- \cong \fH^\vee$.
\end{thm}

\subsubsection{Boundary at the cusps} The cokernel of the natural map $H \to \tH$ is described as the boundary at cusps. Ohta (\cite[Theorem 1.5.5]{cong}) has shown that the module of cusps is free of rank one as a $\Lambda$-module. That is, there is an exact sequence of $\fH$-modules
$$
0 \to H \to \tH \to \Lambda \to 0.
$$
Moreover, there is a canonical element $\zinf \in \tH$ that gives a generator of $\tH/H$. This is proven in \cite[Lemma 4.8]{sharifi}, following \cite[Theorem 2.3.6]{cong} (cf. \cite[Section 6.2.5]{kato-fukaya}).

\subsubsection{Relation between Hecke and Iwasawa algebras} The following is a consequence of the Iwasawa Main Conjecture. See \cite[Section 2.5.3]{fks} for a nice explanation.

\begin{prop}
The natural inclusions $\Lambda \to \fH$ and $\Lambda \to \h$ induce isomorphisms
$$
\Lambda \isoto \fH/\I
$$
and
$$
\Lambda/\xi_\chi \isoto \h/I.
$$
\end{prop}

\subsubsection{Drinfeld-Manin modification} Let $\HDM = \tH \otimes_\fH \h$. By the previous two paragraphs, there is an exact sequence of $\h$-modules
$$
0 \to H \to \HDM \to \Lambda/\xi_\chi \to 0.
$$
By abuse of notation, we let $\zinf \in \HDM$ be the image of $\zinf \in \tH$.

\section{Sharifi's conjecture}\label{sharifi section} 
 In this section, we will discuss some remarkable conjectures that were formulated by Sharifi \cite{sharifi}. Sharifi gave a conjectural construction of a map
$$
\varpi: H^-/IH^- \to X_\chi(1)
$$
and constructed a map
$$
\U : X_\chi(1) \to H^-/IH^-.
$$

\begin{conj}[Sharifi]\label{sharifi conj}
The maps $\U$ and $\varpi$ are inverse isomorphisms.
\end{conj}

We refer to \cite{sharifi} and \cite{kato-fukaya} for the original constructions, and \cite{fks} for a nice survey of the known results. There is also the following weaker version, which appears as \cite[Conjecture 7.1.2]{kato-fukaya}. It allows for the possibility that the $p$-torsion part $(tor)$ of $H^-/IH^-$ is non-zero (note that Conjecture \ref{sharifi conj} implies that $(tor)=0$, and that Sharifi specifically remarks this \cite[Remark, pg. 51]{sharifi}).

\begin{conj}\label{weak sharifi conj}
The maps $\U$ and $\varpi$ are inverses up to torsion. That is, $\U \circ \varpi$ is the identity map on $(H^-/IH^-)/(tor)$ and $\varpi \circ \U$ is the identity map on $X_\chi(1)$.
\end{conj}

The following is \cite[Theorem 7.2.6 (2)]{kato-fukaya}. This result is not needed in the remainder of the paper, except to say that Conjecture \ref{weak sharifi conj} holds in every known example.

\begin{thm}[Fukaya-Kato]\label{fk main thm}
If $\xi_\chi$ has no multiple roots, then Conjecture \ref{weak sharifi conj} is true. In particular, if  $\xi_\chi$ has no multiple roots and $H^-/IH^-$ has no non-zero finite submodule, then Conjecture \ref{sharifi conj} is true.
\end{thm}

The paper \cite{kato-fukaya} also has results on Sharifi's conjecture when $H^-/IH^- \otimes_{\Z_p} \Q_p$ is generated by one element. This is related to Gorenstein conditions on Hecke algebras, the subject of the next section.

\section{Gorenstein Hecke algebras}\label{hecke section} 

In this section we discuss to what extent the Hecke algebras $\h$ and $\fH$ are Gorenstein. The relevant characterization of Gorenstein is the following.

\begin{defn}
Let $k$ be a regular local ring, and let $k \to R$ be a finite, flat ring homomorphism. Then $R$ is {\em Gorenstein} if $\Hom_k(R,k)$ is a free $R$-module of rank $1$. 
\end{defn}

This definition is seen to be equivalent to the usual one from homological algebra, but is more useful for our purposes. In our applications we will take $k=\Lambda$ and $R=\h ,\ \fH$ or their localizations. Asking whether $\h$ or $\fH$ is Gorenstein is the same as asking whether $\h^\vee$ or $\fH^\vee$  is free of rank $1$. This is relevant in light of Theorem \ref{e-s}. 

\subsection{Conditions on $\h$}\label{h section} We consider under what conditions $\h$ is Gorenstein or weakly Gorenstein.

\subsubsection{Gorenstein} The following lemma is explained in \cite[Section 7.2.12]{kato-fukaya}. 

\begin{lem}\label{h lemma 1}
The following are equivalent:
\abcs
\item $\h$ is Gorenstein.
\item $H^-/IH^-$ is cyclic as an $\h$-module.
\item $H^-/IH^-$ is a free $\h/I$-module of rank 1.
\endabcs
\end{lem}
\begin{proof}
Follows from the fact that $H^- \cong \h^\vee$, that $\h^\vee$ is a faithful $\h$-module, and Nakayama's lemma.
\end{proof}

One may ask whether $\h$ is always Gorenstein. The following result is based on ideas of Kurihara \cite{kurihara} and Harder-Pink \cite{harder-pink}, who prove it in the case $N=1$. The result in this form was proven by Ohta. 

\begin{thm}
Suppose that $\h$ is Gorenstein. Then $X_\chi(1)$ is cyclic as a $\Lambda$-module.
\end{thm}
\begin{proof}
This is proven, for example, in \cite[Corollary 4.2.13]{comp2}, where it is the implication ``(ii) $\implies$ (i)''. Note that the proof of ``(ii) $\implies$ (i)'' given there does not require the assumption that $\fH$ is Gorenstein.
\end{proof}

As remarked in Section \ref{greenberg section}, there are examples where $X_\chi$ is not cyclic and therefore where $\h$ is not Gorenstein. One could also ask if the converse holds.

\begin{lem}
Suppose that $H^-/IH^- \cong X_\chi(1)$. Then $\h$ is Gorenstein if and only if $X_\chi(1)$ is cyclic as a $\Lambda$-module.
\end{lem}
\begin{proof}
Follows from Lemma \ref{h lemma 1}.
\end{proof}

In particular, we have the following.

\begin{cor}
Assume $N=1$, and assume Sharifi's Conjecture \ref{sharifi conj} and the Kummer-Vandiver Conjecture \ref{kummer-vandiver}. Then $\h$ is Gorenstein.
\end{cor}

\subsubsection{Weakly Gorenstein} We recall that $\h$ is said to be {\em weakly Gorenstein} if $\h_\p$ is Gorenstein for every prime ideal $\p \in \Spec(\h)$ of height $1$ such that $I \subset \p$. This definition in relevant in light of the following lemma, which appears in \cite[Section 7.2.10]{kato-fukaya}.

\begin{lem}\label{h lemma 2}
The following are equivalent:
\abcs
\item $\h$ is weakly Gorenstein.
\item $(H^-/IH^-) \otimes_{\Z_p} \Q_p$ is cyclic as an $\h/I \otimes_{\Z_p} \Q_p$-module.
\item $(H^-/IH^-) \otimes_{\Z_p} \Q_p$ is a free $\h/I \otimes_{\Z_p} \Q_p$-module of rank 1.
\endabcs
\end{lem}

Fukaya and Kato have the following result on Sharifi's conjecture, assuming that $\h$ is weakly Gorenstein \cite[Theorem 7.2.8 (1)]{kato-fukaya}.

\begin{thm}[Fukaya-Kato]\label{fk h weak thm}
Assume that $\h$ is weakly Gorenstein. Then $\U \colon X_\chi(1) \to (H^-/IH^-)/(tor)$ and $\varpi \colon (H^-/IH^-)/(tor) \to X_\chi(1)$ are isomorphisms.
\end{thm}

Their work also implies the following result on the converse.

\begin{lem}\label{h weakly gor lemma}
Assume that $\coker(\U)$ is finite and that $X_\chi(1) \otimes_{\Z_p} \Q_p$ is cyclic. Then $\h$ is weakly Gorenstein.
\end{lem}
\begin{proof}
If $\coker(\U)$ is finite, then $\U: X_\chi(1) \otimes_{\Z_p} \Q_p \to (H^-/IH^-) \otimes_{\Z_p} \Q_p$ is surjective, so this follows from Lemma \ref{h lemma 2}.
\end{proof}
Since this lemma applies whenever Conjecture \ref{weak sharifi conj} and Conjecture \ref{green conj} hold, we should conjecture that $\h$ is always weakly Gorenstein.

\begin{conj}
The ring $\h$ is weakly Gorenstein.
\end{conj}

\subsection{Conditions on $\fH$} We consider under what conditions $\fH$ is Gorenstein or weakly Gorenstein.

\subsubsection{Gorenstein} Recall that $\HDM^-/H^- \cong \h/I$. In particular, the natural inclusion $I\HDM^- \subset \HDM^-$ lands in $H^-$. The following proposition is proven in \cite{wake}. It can also be proven along the same lines as the proof of Proposition \ref{hecke prop} below.

\begin{prop}
The following are equivalent:
\abcs
\item $\fH$ is Gorenstein.
\item $\HDM^-$ is a free $\h$-module of rank $1$.
\item $I\HDM^- = H^-$.
\endabcs
\end{prop}

It was proven by Ohta that $\fH$ is Gorenstein if $\X_\theta=0$ \cite[Theorem I]{comp2}. Similar results were obtained earlier by Skinner-Wiles \cite{skinner-wiles}. 

The following theorem illustrates the importance of the condition that $\fH$ is Gorenstein. It was first proven by Sharifi, following work of Ohta \cite{cong}.

\begin{thm}[Ohta-Sharifi]\label{H gor implies U iso thm}
Suppose that $\fH$ is Gorenstein. Then $\U$ is an isomorphism.
\end{thm}
\begin{proof}
This is proven in \cite[Proposition 4.10]{sharifi}, where the assumption ``$p \nmid B_{1,\theta^{-1}}$" is not needed in the proof -- all that is needed is the weaker assumption that $\fH$ is Gorenstein.
\end{proof}

Sharifi used this as evidence for his conjecture.  However, it is not true that $\fH$ is always Gorenstein; the following is the main result of the paper \cite{wake}.

\begin{thm}
If $\fH$ is Gorenstein, then either $X_\theta=0$ or $X_\chi=0$. Moreover, there are examples where $X_\theta \ne 0$ and $X_\chi \ne 0$, and so $\fH$ is not always Gorenstein.
\end{thm}

\subsubsection{Weakly Gorenstein} Recall that $\fH$ is said to be {\em weakly Gorenstein} if $\fH_\p$ is Gorenstein for every prime ideal $\p \in \Spec(\fH)$ such that $I_\fH \subset \p$.

\begin{prop} \label{hecke prop}
Let $\fP \subset \Spec(\fH)$ be the set of height $1$ prime ideals $\p$ such that $I_\fH \subset \p$. The following are equivalent:
\abcs
\item \label{first Hecke} $H^-/I\HDM^-$ is finite.
\item As a module over $\h \otimes_{\Z_p} \Q_p$, $\HDM^-/I\HDM^- \otimes_{\Z_p} \Q_p$ is generated by $\zinf$.
\item For any $\p \in \fP$, $(\tH^-)_\p$ is generated by $\zinf$.
\item For any $\p \in \fP$, $(\tH^-)_\p$ is generated by $1$ element.
\item For any $\p \in \fP$, $(\tH^-)_\p$ is free of rank $1$ as an $\fH_\p$-module.
\item $\fH$ is weakly Gorenstein.
\endabcs
\end{prop}
\begin{proof}
\mbox{}
$(1) \implies (2)$: Follows from taking $\otimes_{\Z_p} \Q_p$ in the exact sequence
$$
0 \to H^-/I\HDM^- \to \HDM^-/I\HDM^- \to \HDM^-/H^- \to 0.
$$

$(2) \implies (3)$: Since $\fH/I_\fH \isoto \h/I$, we have that 
$$\HDM^-/I\HDM^- \cong \HDM^- \otimes_\h \h/I \cong \tH^- \otimes_\fH \fH/I_\fH.$$
For any $\p \in \fP$, we have that $p$ is invertible in $\fH_\p$, so $(\HDM^-/I\HDM^-)_\p = (\tH^-/I_\fH \tH^-)_\p$ is generated by $\zinf$. By Nakayama's Lemma, we have (3).

$(3) \implies (4)$: Clear.

$(4) \implies (5)$: By Theorem \ref{e-s}, $(\tH^-)_\p$ is a dualizing module for $\fH_\p$, and so it is faithful. Then, if it is generated by $1$ element, it is free.

$(5) \Longleftrightarrow (6)$: Since $(\tH^-)_\p$ is a dualizing module for $\fH_\p$, this is clear.

$(5) \implies (1)$:  Note that $H^-/I\HDM^-$ is a $\fH/I_\fH$-module. To show (1), it suffices (by Lemma \ref{fin lem}) to show that $H^-/I\HDM^-$ is not supported on any non-maximal prime ideals of $\fH/I_\fH$. Since the non-maximal prime ideals of $\fH/I_\fH$ are exactly the images under $\fH \onto \fH/I_\fH$ of elements of $\fP$, it is enough to show that that $\Supp_\fH(H^-/I\HDM^-) \cap \fP$ is empty.

Let $\p \in \fP$. By $(5)$ we see that $(\tH^-/I_\fH \tH^-)_\p$ is free of rank $1$ as an $(\fH/I_\fH)_\p$-module. But, since $\fH/I_\fH \isoto \h/I$, $(\HDM^-/H^-)_\p$ is also free of rank $1$ as an $(\fH/I_\fH)_\p$-module. Then the natural surjective map 
$$
(\HDM^-/I\HDM^-)_\p=(\tH^-/I_\fH \tH^-)_\p \onto (\HDM^-/H^-)_\p
$$
must be an isomorphism. This implies that the kernel, $(H^-/I\HDM^-)_\p$, is zero. 
\end{proof}

\section{Pairing with cyclotomic units}\label{pairing section}

In this section, we recall some results from \cite{wake} that will be used in the proof on Theorem \ref{main}.

\subsection{The Kummer pairing} As in \cite[Section 3.2]{wake}, we will make use of a pairing between $\X$ and  global units. Let $E$ denote the pro-$p$ part of the closure of the global units in $\varprojlim (\Z[\zeta_{Np^r}]\otimes \Z_p)^\times$.

There is a pairing of $\Z_p[[\Z_{p,N}^\times]]$-modules
$$
[ \ , \ ]_{\mathrm{Kum}}: E \times \X^\#(1) \to \Z_p[[\Z_{p,N}^\times]].
$$
It is essentially defined as the ``$\Lambda$-adic version'' of the pairing
$$
\Z_p[\zeta_{Np^r}]^\times \times \X \to \mu_{p^r} \ ; \ (u, \sigma) \mapsto \frac{\sigma(u^{1/p^r})}{u^{1/p^r}}.
$$
We refer to \cite[Section 3.2]{wake} for the detailed definition.

\subsubsection{The map $\nu$} The Kummer pairing gives a homomorphism of $\Lambda$-modules $E_\theta \to \Hom(\X_\chii,\Lambda^\#(1))$. There is a special element $\cyc \in E_\theta$, namely the image of $(1-\zeta_{Np^r})_r \in \varprojlim(\Z[\zeta_{Np^r}]^\times \otimes \Z_p)$. 

We define $\nu$ to be the image of $\cyc$ under the Kummer pairing. So
$$
\X_\chii \xrightarrow{\nu} \Lambda^\#(1)
$$
is a morphism of $\LpN_\chii$-modules.

The importance of $\nu$ comes from the following lemma, which relates $\nu$ to $X^+$.

\begin{lem}\label{lem1} There exists a natural commutative diagram of $\LpN$-modules with exact rows:
\[\xymatrix{
0 \ar[r] & U_\chii(-1) \ar[r] \ar[d]^-\wr & \X_\chii(-1) \ar[r] \ar[d]^{\nu(-1)} & X_\chii(-1) \ar[r]  & 0 \\
0 \ar[r] & \Lam^\# \ar[r] & \Lam^\# \ar[r] &  \Lam^\#/\iota \xi_\chii \ar[r] & 0.}\]
Let $\bnu:X_\chii(-1) \to \Lam^\#/\iota \xi_\chii$ be the induced map, and let $C$ denote $\coker(\bnu)$. Then we have an equality of characteristic ideals
$$
\Ch(X_\theta)=\Ch(C^\#).
$$
\end{lem}
\begin{proof}
The existence of the commutative diagram is from \cite[Lemma 4.5]{wake} (note that the element denoted by $\xi_\chii$ in \cite{wake} would be denoted $\iota \tau \xi_\chii$ here).

Let $C=\coker(\bnu)$. By [{\em loc. cit.}, Proposition 4.8 (1) and Lemma 4.6], there is an exact sequence
$$
0 \to \E^1(C) \to \Lambda/\xi_\chii \to \X_\theta \to X_\theta \to 0.
$$
Since $\Ch(\X_\theta) = (\xi_\chii)$, we have $\Ch(E^1(C))=\Ch(X_\theta)$.  We claim that $\Ch(E^1(C))=\Ch(C^\#)$. This follows from \cite[Proposition 5.5.13, pg. 319]{N-S-W}, as in the proof of Proposition \ref{X and X prop} above.
\end{proof}

\subsubsection{The map $\nu'$} Let $\nu': \X_\chii \to (\Lambda/\xi_\chi)^\#(1)$ be the composite
$$
\X_\chii \xrightarrow{\nu} \Lambda^\#(1) \to  \Lambda^\#(1)/\iota \tau \xi_\chi\Lambda^\#(1) = (\Lambda/ \xi_\chi)^\#(1).
$$

\begin{lem}\label{lem2}
We have that $\coker(\nu')$ is finite if and only if $X_\theta/ \xi_\chi X_\theta$ is finite. Moreover, $\coker(\nu')=0$ if and only if $X_\chi=0$ or $X_\theta=0$.
\end{lem}
\begin{proof}
Let $C = \coker(\bnu)$. From Lemma \ref{lem1}, we see that $C \cong \coker(\nu(-1))$. We see from the definition of $\nu'$ that 
$$
\coker(\nu') \cong C(1)/\tau^{-1} \iota \xi_\chi C(1) = (C/\iota \xi_\chi C)(1).
$$
Note that since $(C/ \iota \xi_\chi C)^\# \cong C^\#/ \xi_\chi C^\#$, we have that $\coker(\nu')$ is finite if and only if $C^\#/ \xi_\chi C^\#$ is finite.

Recall from Lemma \ref{lem1} that  $\Ch(C^\#)= \Ch(X_\theta)$. We apply Lemma \ref{lambda lem} to the case $M=\Lambda/\xi_\chi$, $N=C^\#$, and $N'=X_\theta$ to get that $C^\#/ \xi_\chi C^\#$ is finite if and only if $X_\theta/ \xi_\chi X_\theta$ is finite. This completes the proof of the first statement.

For the second statement, notice that if $X_\chi=0$, then $\xi_\chi$ is a unit, so $\coker(\nu')=0$. If $X_\chi \ne 0$, then $\xi_\chi$ is not a unit and, by Nakayama's lemma, $\coker(\nu')=0$ if and only if $C=0$. It remains to prove that $C=0$ if and only if $X_\theta=0$, and this follows from \cite{wake}. Indeed, if $X_\theta=0$, then \cite[Proposition 4.8 (2)]{wake} implies that $C=0$. Conversely, if $C=0$, then \cite[Proposition 4.8 (1) and Corollary 4.7]{wake} together imply that $X_\theta$ is finite, and then we can apply \cite[Proposition 4.8 (2)]{wake} to conclude that $X_\theta=0$.
\end{proof}

\subsection{The map $\nu$ as an extension class}\label{up and theta} In \cite[Section 9.6]{kato-fukaya}, Fukaya and Kato give an interpretation of $\nu$ as an extension class. We review this here, and refer to \cite[Section 2.3]{wake} for more details.

There is an exact sequence
\begin{equation}\label{exten}
0 \to H^+/IH^+ \to \HDM/K \to \HDM/H \to 0
\end{equation}
where $K$ is the kernel of the natural map $H \to H^+/IH^+$  (and so $K \cong H^- \oplus IH^+$ as $\h$-modules). It can be shown that $H \to H^+/IH^+$ respects the $G_\Q$-action (\cite{sharifi}, cf. \cite[Proposition 6.3.2]{kato-fukaya}), and so (\ref{exten}) is an extension of $\HDM/H$ by $H^+/IH^+$ as $\h[G_\Q]$-modules. By considering this extension as a Galois cocycle, we obtain a homomorphism of $\LpN_\chii$-modules
$$
\Theta: \X_{\chii} \to (\Lambda/ \xi_\chi)^\#(1).
$$
By \cite[Theorem 9.6.3]{kato-fukaya}, we have the following (cf. \cite[Proposition 3.4]{wake}).

\begin{thm}\label{nu=theta thm}
We have $\nu'=\Theta$.
\end{thm} 

\section{Relationship between the Hecke and Iwasawa sides}\label{proof section}

The goal of this section is to complete the proof of Theorem \ref{main}. 

\subsection{The key diagram} First we consider a commutative diagram coming from the maps of Section \ref{pairing section}. Let $\nu'' = (\nu')^\#(1): \X_\chii^\#(1) \to \Lambda/\xi_\chi$ so that $\nu''$ is a map of $\Lambda$-modules. Then we have the diagram of $\Lambda$-modules:
\begin{equation*}
\xymatrix{
\X_\chii^\#(1) \otLam X_\chi(1) \ar[r]^-{\nu'' \otimes 1} \ar[d]^-{\Phi} & \Lambda/\xi_\chi \otLam X_\chi(1) \ar[d]^-\U \\
H^-/IH^- \ar@{=}[r]  & H^-/IH^-,
}
\end{equation*}
where $\Phi = \Theta^\#(1) \otimes \U$. It is commutative by Theorem \ref{nu=theta thm}.  This is a slight reformulation of the diagram (*) in \cite[Section 1.3]{wake}. We record the result of applying the Snake Lemma to this diagram as a lemma.

\begin{lem}\label{sequence lem} There is an exact sequence:
\begin{equation*}
\ker(\Phi) \to \ker(\Upsilon) \to  \coker(\nu'')\otLam X_\chi(1) \to \coker(\Phi) \to \coker(\U) \to 0.
\end{equation*}
\end{lem}

\subsubsection{Some lemmas} We prove some lemmas.

\begin{lem}\label{coker phi lem}
We have $\coker(\Phi)=H^-/I\HDM^-$.
\end{lem}
\begin{proof}
This is a slight reformulation of \cite[Proposition 2.2]{wake}, which states that the image of $\Phi$ is $I \zinf$. Since $\zinf$ generates $\HDM^-/H^-$, we see that the image of $I \zinf$ in  $H^-/IH^-$ is the same as the image of $I\HDM^-$ in $H^-/IH^-$.
\end{proof}

\begin{lem}\label{lem4}
We have that $\coker(\nu'') \otLam X_\chi(1)$ is finite if and only if $X_\theta/ \xi_\chi X_\theta$ is finite.
\end{lem}
\begin{proof} Let $C'' = \coker(\nu'')$. Since $\nu''=(\nu')^\#(1)$, it is clear that $C''$ is finite if and only if $\coker(\nu')$ is finite.
By Lemma \ref{lem2} it suffices to show that $C'' \otLam X_\chi(1)$ is finite if and only if $C''$ is finite. 

Now apply Lemma \ref{lambda lem1} to $M=C''$ and $N=X_\chi(1)$ to get that $C'' \otLam X_\chi(1)$ is finite if and only if $C''/\Ch(X_\chi(1)) C''$ is finite. However,  $\Ch(X_\chi(1))=(\xi_\chi)$, which annihilates $C''$. So $C''=C''/\Ch(X_\chi(1)) C''$ and the lemma follows.
\end{proof}

\begin{prop}\label{sha prop}
Suppose that $\coker(\U)$ is finite. Then $\Upsilon$ is injective.
\end{prop}
\begin{proof}
It is well-known that 
$$
\Ft(H^-/IH^-) \subset (\xi_\chi)
$$
(cf. \cite[Section 7.1.3]{kato-fukaya}). 

We apply Lemma \ref{fitt lem} to the case $M=X_\chi(1)$, $N=H^-/IH^-$ and $f=\Upsilon$. It says that if $\coker(\U)$ is finite, then $\ker(\Upsilon)$ is finite. But $X_\chi(1)$ has no finite submodule, so the result follows.
\end{proof}

\subsubsection{The proof of Theorem \ref{main}} We can now prove Theorem \ref{main}, which we restate here for convenience. 

\begin{main}
Both  $\coker(\U)$ and $X_\theta/ \xi_\chi X_\theta$ are finite if and only if $\fH$ is weakly Gorenstein.
\end{main}
\begin{proof}
First assume that $\coker(\U)$ and $X_\theta/ \xi_\chi X_\theta$ are finite. By Lemma \ref{lem4} we have that $\coker(\nu'') \otLam X_\chi(1)$ is finite. By Lemma \ref{sequence lem}, we see that $\coker(\Phi)$ is finite. By Lemma \ref{coker phi lem}, $H^-/I\HDM^-$ is finite. By Proposition \ref{hecke prop}, we have that $\fH$ is weakly Gorenstein.

Now assume that $\fH$ is weakly Gorenstein. Then, as above, $\coker(\Phi)$ is finite. From Lemma \ref{sequence lem}, we see that $\coker(\Upsilon)$ is finite. By Proposition \ref{sha prop}, we see that $\ker(\Upsilon)=0$. Again using Lemma \ref{sequence lem} we see that $\coker(\nu'') \otLam X_\chi(1) \subset \coker(\Phi)$, and so $\coker(\nu'') \otLam X_\chi(1)$ is finite. By Lemma \ref{lem4}, we have that $X_\theta/ \xi_\chi X_\theta$ is finite.
\end{proof}

\section{Application to Sharifi's Conjecture}\label{application}

 For a finitely generated $\Lambda$-module $M$, let $$d_{\m_\Lambda}(M)=\dim_{\Lambda/\m_\Lambda}(M/\m_\Lambda M).$$
Note that, by Nakayama's lemma, $d_{\m_\Lambda}(M)$ is minimal number of generators of $M$. In particular,  $d_{\m_\Lambda}(M)=0$ if and only if $M=0$.

 We can now prove Theorem \ref{application cor}, which we restate here for convenience.

\begin{app}
Assume that $X_\theta \ne 0$ and that $\h$ is weakly Gorenstein.

Then we have 
$$d_{\m_\Lambda}(H^-/I\HDM^- ) \ge d_{\m_\Lambda}(X_\chi(1))$$
with equality if and only if $\U$ is an isomorphism.

If, in addition, $\# (X_\theta)= \# (\Lambda/\m_\Lambda)$, then $\U$ is an isomorphism if and only if 
$$
\#(H^-/I\HDM^-)=\# (\Lambda/\m_\Lambda)^{d_{\m_\Lambda}(X_\chi(1))}.
$$
\end{app}
\begin{proof}
By Theorem \ref{fk h weak thm}, we have that $\coker(\U)$ is finite. By Proposition \ref{sha prop}, we have $\ker(\U)=0$. By Lemma \ref{sequence lem}, we have an exact sequence
$$
0 \to \coker(\nu'')\otLam X_\chi(1) \to \coker(\Phi) \to \coker(\Upsilon) \to 0.
$$
Theorem \ref{fk h weak thm} implies that $\coker(\Upsilon) \isoto (tor) \to \coker(\Phi)$ gives a spitting of this sequence. This gives us an isomorphism
$$
H^-/I\HDM^- = \coker(\Phi) \cong (\coker(\nu'')\otLam X_\chi(1)) \oplus \coker(\Upsilon).
$$
and so
\begin{align*}
d_{\m_\Lambda}(H^-/I\HDM^-) & =  d_{\m_\Lambda}(\coker(\nu'')\otLam X_\chi(1)) + d_{\m_\Lambda}(\coker(\Upsilon)) \\
			   & =  d_{\m_\Lambda}(\coker(\nu''))d_{\m_\Lambda}(X_\chi(1)) + d_{\m_\Lambda}(\coker(\Upsilon)).
\end{align*}

We claim that, in fact,
$$d_{\m_\Lambda}(H^-/I\HDM^-)=d_{\m_\Lambda}(X_\chi(1)) +d_{\m_\Lambda}(\coker(\Upsilon)),$$
from which the first statement of the theorem follows.

To prove the claim, note that it is clear if $d_{\m_\Lambda}(X_\chi(1))=0$. Now assume $d_{\m_\Lambda}(X_\chi(1)) \ne 0$. Then we claim that $d_{\m_\Lambda}(\coker(\nu''))=1.$ Indeed, since $\coker(\nu'')$ is cyclic it suffices to show $\coker(\nu'') \ne 0$. But since $X_\chi(1) \ne 0$, Lemma \ref{lem2} implies that $\coker(\nu'') \ne 0$ if and only if $X_\theta \ne 0$, which we are assuming. This completes the proof of the claim and of the first statement.

For the second statement, notice that the assumption can only occur if $X_\theta \cong \Lambda/\m_\Lambda$. By \cite[Proposition 4.8]{wake}, this implies that $\coker(\bnu)^\# \cong  \Lambda/\m_\Lambda$. As in Lemma \ref{lem2}, where we computed $\coker(\nu')$ in terms of $\coker(\bnu)$, we compute
$$
\coker(\nu'') \cong \coker(\bnu)^\#/ \xi_\chi \coker(\bnu)^\#,
$$
so 
$$
\coker(\nu'') \cong \left \{
\begin{array}{cc}
\Lambda/\m_\Lambda & \text{ if } X_\chi(1) \ne 0 \\
0 & \text{ if } X_\chi(1) = 0.
\end{array}
\right.
$$
In either case,
$$
\coker(\nu'') \otLam X_\chi(1) \cong (\Lambda/\m_\Lambda)^{d_{\m_{\Lambda}}(X_\chi(1))},
$$
and the statement follows from the established isomorphism
$$
H^-/I\HDM^- \cong (\coker(\nu'')\otLam X_\chi(1)) \oplus \coker(\Upsilon).
$$
\end{proof} 

\appendix
\section{Some commutative algebra}\label{lambda section}
We review some lemmas from commutative algebra that are used in the body of the paper. The results of this appendix are well-known; we include them for completeness.

\subsection{Finite modules} We begin with a review of some generalities about finite modules (meaning modules of finite cardinality). Let $(A,\m)$ be a noetherian local ring, and assume that the residue field $A/\m$ is finite. For an $A$-module $M$, we use the notation $\Supp_A(M)$ for the set $\{ \p \in \Spec(A) \ | \ M_\p \ne 0 \}$.

\begin{lem}\label{fin lem}
Let $M$ be a finitely generated $A$-module. The following are equivalent:
\abcs
\item $\m^n M = 0$ for some $n$.
\item $M$ is finite.
\item $M$ is an Artinian $A$-module.
\item $\Supp_A(M) \subset \{\m\}$.
\endabcs
\end{lem}
\begin{proof}
For $(4) \implies (1)$, since $M$ is finitely generated, it is enough to prove the case where $M \cong A/I$ for an ideal $I$. By (4) we have that $\Spec(A/I) \subset \{\m\}$. This implies that $A/I$ is Artinian, which implies that $\m^n (A/I) = 0$ for some $n$. The implications $(1) \implies (2) \implies (3) \implies (1) \implies (4)$ are clear.
\end{proof}

\begin{cor}\label{tprod cor}
Suppose $M$ and $N$ are finitely generated $A$-modules. Then $M \otimes_A N$ is finite if and only if $\Supp_A(N) \cap \Supp_A(M) \subset \{\m\}$.
\end{cor}
\begin{proof}
This is clear from Lemma \ref{fin lem} since $$\Supp_A(M \otimes_A N) = \Supp_A(N) \cap \Supp_A(M).$$
\end{proof}

\subsection{$\Lambda$-modules} Let $\Lambda$ be a noetherian regular local ring of dimension 2 with finite residue field. For example, $\Lam = \sO[[T]]$ where $\sO$ is the valuation ring of a finite extension of $\Q_p$.

\subsubsection{Characteristic ideals} For a finitely generated torsion $\Lambda$-module $M$, define the characteristic ideal of $M$ to be
$$
\Ch(M)=\prod_\p \p^{l_\p(M)}
$$
where $\p$ ranges over all height 1 primes of $\Lambda$, and $l_\p(M)$ is the length of $M_\p$ as a $\Lambda_\p$ module. Note that $l_\p(M) >0 $ if and only if $\p \in \Supp_\Lambda(M)$.

It follows from the definition that $\Ch$ is multiplicative on exact sequences, and that $\Ch(M)$ is a principal ideal. By Lemma \ref{fin lem}, $\Ch(M)=\Lambda$ if and only if $\Lambda/\Ch(M)$ is finite if and only if $M$ is finite. We have the following consequence of Corollary \ref{tprod cor}.

\begin{lem}\label{lambda lem1}
Let $N$ and $M$ be a finitely generated $\Lambda$-modules and suppose that $N$ is torsion. Then $M \otimes_\Lambda N$ is finite if and only if $M/\Ch(N)M$ is finite.
\end{lem}
\begin{proof}
This is clear from Corollary \ref{tprod cor} since $\Supp_\Lambda(N)=\Supp_\Lambda(\Lambda / \Ch(N))$.
\end{proof}

In the body of the paper, we often use Lemma \ref{lambda lem1} in the following form.

\begin{lem}\label{lambda lem}
Let $N$, $N'$, and $M$ be a finitely generated $\Lambda$-modules and suppose that $N$ and $N'$ are torsion and that $\Ch(N)=\Ch(N')$. Then $M \otimes_\Lambda N$ is finite if and only if $M \otimes_\Lambda N'$ is finite.
\end{lem}

\subsubsection{Fitting ideals} Let $R$ be a commutative, noetherian ring. For a finitely generated $R$-module $M$, we define $\mathrm{Fitt}_R(M) \subset R$, the Fitting ideal of $M$, as follows. Let 
$$
R^m \xrightarrow{A} R^n \to M \to 0
$$
be a presentation of $M$. Then $\mathrm{Fitt}_R(M)$ is defined to be the $R$-module generated by all the $(n,n)$-minors of the matrix $A$. This does not depend on the choice of resolution (cf. \cite[Appendix]{mazur-wiles}). 

The following lemma follows from the independence of resolution.
\begin{lem}\label{basic fitt lem}
If $\phi: R \to R'$ is a ring homomorphism, and $M$ is an $R$-module, then $$\mathrm{Fitt}_{R'}(M \otimes_R R') \subset R'$$ is the ideal generated by $\phi (\mathrm{Fitt}_R(M)).$
\end{lem}

We consider the case $R=\Lambda$. The following relation to $\Ch$ is a well-known.

\begin{lem}\label{wk lem}
If $M$ is finitely generated and torsion, then $\Ch(M)$ is the unique principal ideal such that $\Ft(M) \subset \Ch(M)$ and $\Ch(M)/\Ft(M)$ is finite. 
\end{lem}
\begin{proof}
Using Lemma \ref{basic fitt lem}, we see that for any prime $\p$ of $\Lambda$,$$\mathrm{Fitt}_{\Lambda_\p}(M_\p)=\Ft(M)_\p .$$
Using the fact that $\Lambda_\p$ is a DVR for a height 1 prime $\p$, we have $\mathrm{Fitt}_{\Lambda_\p}(M_\p) = \p^{l_\p(M)}$ by the structure theorem for modules over a PID. We have then that $\Ft(M)_\p=\Ch(M)_\p$ for all $\p$ of height 1. Lemma \ref{fin lem} then implies that $\Ch(M)/\Ft(M)$ is finite.

For the uniqueness, suppose $\Ft(M) \subset (f)$ has finite quotient. Then $\Ft(M)_\p = (f)_\p$ for each height one prime $\p$. This determines the prime factorization of $f$.
\end{proof}

Using this relation, we can deduce the  following.

\begin{lem}\label{fitt lem}
Let $M$ and $N$ be two finitely generated torsion $\Lambda$-modules. Assume that $\Ft(N) \subset \Ch(M)$. If a morphism $f: M \to N$ has finite cokernel, then it has finite kernel.
\end{lem}
\begin{proof}
Indeed, if $f$ has finite cokernel, then 
$$
\Ch(M) = \Ch(\ker(f))\Ch(N),
$$
and so
$$
\Ch(M) \subset \Ch(N).
$$
Since $\Ch(N)/\Ft(N)$ is finite, this implies that $\Ch(M)/\Ft(N)$ is finite. By Lemma \ref{wk lem}, this implies that $\Ch(N)=\Ch(M)$. The result follows by multiplicativity of characteristic ideals.
\end{proof}

Department of Mathematics, Eckhart Hall, University of Chicago, Chicago, Illinois 60637

email: pwake@math.uchicago.edu
\end{document}